\documentclass[10pt]{amsart}
\usepackage{amssymb}
\usepackage{amsmath,amssymb,amsfonts,amsthm,
latexsym, amscd, amsfonts, epsfig,epic,psfrag,colordvi}
\usepackage{mathrsfs}
\usepackage{color}
\usepackage{eucal}
\usepackage{eufrak}
\usepackage[all]{xypic}
\usepackage{xspace}

\theoremstyle{plain}
\newtheorem{theorem}{Theorem}

\newtheorem{lemma}{Lemma}

\theoremstyle{definition}

\theoremstyle{example}

\theoremstyle{remark}

\numberwithin{equation}{section}
\usepackage{colordvi}

\begin{document}

\begin{center}
{\bf\Large Random induced subgraphs of Cayley graphs induced by transpositions}\\
\vspace{15pt} Emma Yu Jin\footnote{The work of this author has been
supported by the Alexander von Humboldt Foundation by a postdoctoral
research fellowship.} and  Christian M. Reidys\footnote{Author to
whom correspondence should be addressed.}
\end{center}

\begin{minipage}[t]{0.4\linewidth}
\centering
$^1$Department of Computer Science\\
University of Kaiserslautern,\\
67663 Kaiserslautern, Germany\\
Email: jin@cs.uni-kl.de
\end{minipage}
\begin{minipage}[t]{0.6\linewidth}
\centering
$^2$Department of Mathematics and Computer Science\\
University of Southern Denmark,\\
Campusvej 55, DK-5230 Odense M, Denmark\\
Email: duck@imada.sdu.dk
\end{minipage}

\thanks{}
\date{September, 2009}
\begin{abstract}
In this paper we study random induced subgraphs of Cayley graphs of
the symmetric group induced by an arbitrary minimal generating set of
transpositions. A random induced subgraph of this Cayley graph is
obtained by selecting permutations with independent probability,
$\lambda_n$. Our main result is that for any minimal generating set of
transpositions, for probabilities $\lambda_n=\frac{1+\epsilon_n}{n-1}$
where $n^{-\frac{1}{3}+\delta}\le \epsilon_n<1$ and $\delta>0$,
a random induced subgraph has a.s.~a unique largest component of size
$ (1+o(1))\cdot x(\epsilon_n)\cdot\frac{1+\epsilon_n}{n-1}\cdot n!$.
Here $x(\epsilon_n)$ is the survival probability of Poisson branching
process with parameter $\lambda=1+\epsilon_n$.
\end{abstract}
\section{Introduction}

One central problem arising in parallel computing is to determine an
optimal linkage of a given collection of processors. A particular
class of processor linkages with point-to-point communication links
are static interconnection networks. The latter are widely used for
message-passing architectures. A static interconnection network can
be represented as a graph. The binary $n$-cubes, $Q_{2}^{n}$,
\cite{Ajtai:82,Padmanabhan:89} are a particularly well-studied class
of interconnection networks
\cite{Chlebus:96,Esfahanian:93,Fan:05,Sharma:06}.

Akers {\it et al.} \cite{Akers:87} observed the deficiencies of
$n$-cubes as models for interconnection networks and proposed an
alternative: the Cayley graph of the permutation group induced by
the $(n-1)$ star-transpositions $(1\,i)$, which was denoted by
$\Gamma(S_n,P_n)$. Pak \cite{Pak:99}
studied minimal decompositions of a particular permutation via
star-transpositions and Irving {\it et al.} \cite{Irving} extended
his results. The star-graph $\Gamma(S_n,P_n)$ is in many aspects
superior to $n$-cubes \cite{Ajtai:82,Padmanabhan:89}. Some
properties of star-graphs studied in
\cite{Hsieh:00,Hsieh:01,Hsieh:012,Hsieh:05,Jwo:91,Li:04} were
cycle-embeddings and path-embeddings. Diameter and fault diameter of
star-graphs were computed by Akers {\it et al.}
\cite{Akers:87,Latifi:93,Rouskov:96} and Lin {\it et al.}
\cite{Lin:08} analyzed diagnosability. An alternative to $n$-cubes
as interconnection networks are the bubble-sort graphs
\cite{Akers:89}, studied by Tchuente \cite{Tchuente:82}. The
bubble-sort graph is the Cayley graph of the permutation group
induced by all $n-1$ canonical transpositions $(i\,\,i+1)$, denoted
by $\Gamma(S_n,B_n)$.

Recently, Araki \cite{Araki:2006} brought the attention to a
generalization of star- and bubble-sort graphs, the Cayley graph
generated by all transpositions \cite{Cayley}. The latter has direct connections
to a problem of interest in computational biology: the evolutionary
distances between species based on their genome order in the Cayley
graph of signed permutations generated by reversals. A reversal is a
special permutation that acts by flipping the order as well as the
signs of a segment of genes. Hannenhalli and Pevzner
\cite{Pevzner:95} presented an algorithm computing minimal number of
reversals needed to transform one sequence of distinct genes into a
given signed permutation. For distant genomes, however, it is
well-known, that the true evolutionary distance is generally much
greater than the shortest distance
\cite{Wang,Caprara,Pevzner,Berestycki:06}. In order to obtain a more
realistic estimate of the true evolutionary distance, the expected
reversal distance was shifted into focus. Its computation, however,
has proved to be hard and motivated models better
suited for computation. Point in case is the work of Eriksen {\it et
al.} \cite{Eriksen}, where the authors derive a closed formula for
the expected transposition distance and subsequently show how to use
it as an approximation of the expected reversal distance. Berestycki 
and Durrett \cite{Berestycki:07} studied the shortest distance of 
random walks over Cayley graphs generated by all transpositions and 
canonical transpositions, respectively, and compared the shortest distance with 
the expected distance \cite{Eriksen}.

The theory of random graphs was pioneered by Erd\"{o}s and R\'{e}nyi
in the late $1950$s \cite{Erdos:59,Erdos:60}, who analyzed the phase
transition of $G(n,p_n)$, the random graph containing $n$ vertices
in which an edge $\{i,j\}$ is selected with independent probability
$p_n$. For $p_n=\frac{c}{n}$ and $c<1$, the largest component in
$G(n,p_n)$ is a.s.~of size $O(\log n)$.
For $p_n=\frac{1+\theta\cdot n^{-\frac{1}{3}}}{n}$, where $\theta>0$,
a.s.~a largest component of size $O(n^{\frac{2}{3}})$ emerges.
For $p_n=\frac{c}{n}$ and $c>1$, we have a.s.~a unique largest
component of size $O(n)$ and all other components are smaller than
$O(\log n)$.
Erd\"{o}s and R\'{e}nyi's construction of the giant component
\cite{Erdos:59,Erdos:60} has motivated Lemma~\ref{L:minigene},
which assures the existence of certain subtrees of size $\lfloor
\frac{1}{4}n^{\frac{2}{3}}\rfloor$.
For a review of Erd\"{o}s-R\'{e}nyi random graph theory, see
Durrett \cite{Durrett:07} or van der Hofstad \cite{Hofstad:10}.

In this paper we study a subgraph of the Cayley graph generated by
all transpositions, the Cayley graph $\Gamma(S_n,T_n)$, where $T_n$
is a minimal generating set of transpositions.
Setting $T_n=P_n$ and $T_n=B_n$ we can recover the star- and the
bubble-sort graph as particular instances. We study structural properties
of $\Gamma(S_n,T_n)$ in terms of the random graph obtained by selecting
permutations with independent probability. The main result of this paper is
\begin{theorem}\label{T:main}
Let $\lambda_n=\frac{1+\epsilon_n}{n-1}$, where
$n^{-\frac{1}{3}+\delta}\le \epsilon_n<1$ and $\delta>0$. Let $T_n$
be a minimal generating set of transpositions and let
$\Gamma_n$ denote the random induced subgraph of
$\Gamma(S_n,T_n)$, obtained by independently selecting each
permutation with probability $\lambda_n$. Then $\Gamma_n$ has a.s.~a unique
giant component, $C_n^{(1)}$, whose size is given by
\begin{equation}
\vert C_n^{(1)}\vert=
(1+o(1))\cdot x(\epsilon_n)\cdot\frac{1+\epsilon_n}{n-1}\cdot n!,
\end{equation}
where $x(\epsilon_n)>0$ is the survival probability of a Poisson
branching process with parameter $\lambda=1+\epsilon_n$ and also the unique
positive root of $e^{-(1+\epsilon_n)y}=1-y$.
Particularly, if $n^{-\frac{1}{3}+\delta}
\le\epsilon_n=o(1)$, then we have $x(\epsilon_n)=(2+o(1))\epsilon_n$.
\end{theorem}

\begin{figure*}
\includegraphics[width=0.5\textwidth]{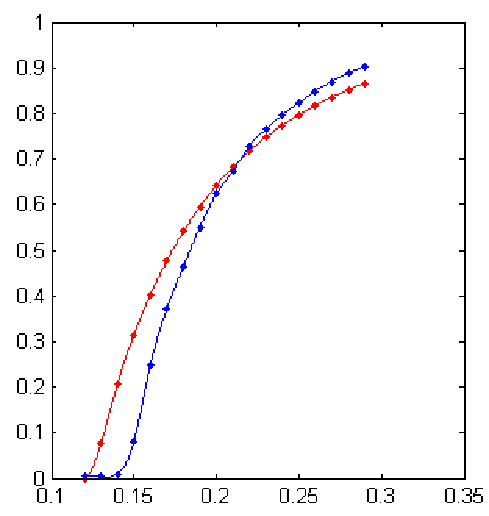}
\caption{The evolution of the giant component in random induced
subgraphs of $\Gamma(S_9,P_9)$. We display the relative size of the
giant component $\frac{\vert C_9^{(1)}\vert}{\vert\Gamma_9\vert}$ as
a function of $\lambda_9=(1+\epsilon)/8$ as data-curve (blue) versus
the growth predicted by Theorem~\ref{T:main} (red).} \label{F:ex}
\end{figure*}

In contrast to vertex-induced random graphs, edge-induced random graphs
have been studied quite extensively.
Random induced subgraphs of $n$-cubes \cite{Bollobas:91,Reidys:08rand}.
as well as $G(n,p_n)$ and random induced subgraphs of $\Gamma(S_n,T_n)$
exhibit a giant component for very small vertex selection probabilities.
One might speculate that the critical probability
$p_n=\frac{1+\theta\cdot n^{-\frac{1}{3}}}{n}$ is determined by the size
of the generator set. Note that $\vert T_n\vert=n-1$ holds for any
minimal generating set of transpositions and the size of the generator 
set for $n$-cube is $n$. Specific properties of $n$-cubes,
like for instance, the isoperimetric inequality \cite{Harper:66b},
do not play a key role for establishing the existence of the giant
component. The isoperimetric inequality depends on an inductive argument
using particular properties of a linear ordering of the vertices
of an $n$-cube. This induction cannot be carried out for Cayley graphs over
canonical transpositions.
In this paper any argument involving (vertex) boundaries follows from a
generic estimate of the vertex boundary in Cayley graphs due to Aldous
\cite{Aldous:87,Babai:91b}.

The paper is organized as follows: after introducing in
Section~\ref{S:2} our notation and some basic facts about branching
processes, we analyze in Section~\ref{S:3} vertices contained in
polynomial size subcomponents. The strategy is similar to that in
\cite{Reidys:08rand}, where first a specific branching process is
embedded (for its first $\lfloor \frac{1}{4}n^{\frac{2}{3}}\rfloor$
steps) into $\Gamma(S_n,T_n)$. It is its survival probability that
provides a lower bound on the probability that a given vertex is
contained in a subcomponent of arbitrary, polynomial size. In
Section~\ref{S:4} we ``sandwich'' this bound by showing that there
are many vertices in ``small'' components. Only here we use
$\epsilon<1$. In Section~\ref{S:5} we show that there are many
vertex disjoint paths between certain splits of permutations. The
a.s.~existence of the giant component follows using the ideas of
Ajtai {\it et al.} \cite{Ajtai:82}.

\section{Background and notation}\label{S:2}

Let $S_n$ denote the symmetric group over $[n]$. We write a permutation
$\pi\in S_n$ as an $n$-tuple $(x_1,x_2,\cdots,x_n)$, i.e.,
$$
\left(\begin{array}{cccc}
  1 & 2 & \cdots & n\\
  x_1 & x_2 & \cdots & x_n
\end{array}\right)=(x_1,x_2,\cdots,x_n).
$$
Particularly we use $(i\,j)$ to briefly denote the transpositions that
merely interchange the elements at positions $i$ and $j$ of the identity
permutation. Plainly, we have
\begin{equation}
(x_1,\cdots,x_i,x_{i+1},\cdots,x_{j-1},x_j,\cdots,x_n)\cdot(i\,j)=
(x_1,\cdots,x_j,x_{i+1},\cdots,x_{j-1},x_i,\cdots, x_n).
\end{equation}
Furthermore, we set $((x_1,\cdots,x_n))_m=x_m$ i.e.~extracting the $m$-th
coordinate.
Let $T_n\subset S_n$ be a minimal generating set of transpositions.
We consider the Cayley graph $\Gamma(S_n,T_n)$, having vertex set
$S_n$ and edges $\{v,v'\}$ where $v^{-1}\cdot v'\in T_n$.
For $v,v'\in S_n$, let $d(v,v')$ be the minimal
number of $T_n$-transpositions by which $v$ and $v'$ differ.
For $A\subset S_n$ we set $\text{\sf B}(A,j)=\{v\in S_n\mid \exists \,
\alpha\in A;\,d(v,\alpha)\le j\}$ and
$\text{\sf d}(A,i) = \{v\in S_n\setminus A\mid \exists \,
\alpha\in A;\, d(v,\alpha)=i\}$
and call $\text{\sf B}(A,j)$ and $\text{\sf d}(A)=\text{\sf d}(A,1)$
the ball of radius $j$ around $A$ and the vertex boundary of $A$ in
$\Gamma(S_n,T_n)$. If $A=\{\alpha\}$ we simply write
$\text{\sf B}(\alpha,j)$. Let $D,E\subset S_n$, we call $D$
$\ell$-dense in $E$ if $\text{\sf B}(\sigma,\ell)\cap
D\neq\varnothing$ for any $\sigma\in E$.
Let ``$\le$'' be the following linear order over $\Gamma(S_n,T_n)$
\begin{equation}\label{E:order}
\sigma \le \tau \quad \Longleftrightarrow \quad \sigma=\tau \
{\mbox {or}} \
\sigma<_{\text{\rm lex}} \tau,
\end{equation}
where $<_{\text{\rm lex}}$ denotes the lexicographical order. Any
notion of minimal or smallest element in a subset $A\in S_n$
refers to eq.~(\ref{E:order}).

Let $\Gamma_{\lambda_n}(S_n,T_n)$ be the probability space (random graph)
consisting of $\Gamma(S_n,T_n)$-subgraphs, $\Gamma_n$, induced by selecting
each $\Gamma(S_n,T_n)$-vertex with independent probability
$\lambda_n$.
A property $\text{\sf M}$ is a subset of induced subgraphs of
$\Gamma(S_n,T_n)$ closed under graph isomorphisms. The terminology
``$\text{\sf M}$ holds a.s.'' is equivalent to
$\lim_{n\to\infty}{\mathbb{P}}(\text{\sf M})=1$. A component of
$\Gamma_n$ is a maximal, connected, induced
$\Gamma_n$-subgraph, $C_n$. The largest
$\Gamma_n$-component is denoted by $C_n^{(1)}$. We write
$x_n\sim y_n$ if and only if (a) $\lim_{n\to\infty}x_n/y_n$ exists
and (b) $\lim_{n\to\infty}x_n/y_n=1$. We set $g(n)=o(f(n))$ if and only
if $g(n)/f(n)\rightarrow 0$.
A largest $\Gamma_n$-component $C_n^{(1)}$ is called giant if it is
unique, i.e.~any other component, $C_n$, satisfies $\vert C_n
\vert=o(\vert C_n^{(1)} \vert)$.

We furthermore write $g(n)=O(f(n))$ as $n\rightarrow \infty$ if and only if
$\frac{g(n)}{f(n)}$ is bounded as $n\rightarrow\infty$, i.e., for
arbitrary $M>0$, there exists a constant $C$ (independent of $M$) such that
for all $n>M$, $\left|\frac{g(n)}{f(n)}\right|\le C$.

Let $Z_n=\sum_{i=1}^n \xi_i$ be a sum of mutually independent indicator random
variables (r.v.), $\xi_i$ having values in $\{0,1\}$. Then we have,
\cite{Chernoff:52}, for $\eta>0$ and
$c_\eta=\min\{-\ln(e^{\eta}[1+\eta]^{-[1+\eta]}),
\frac{\eta^2}{2}\}$
\begin{equation}\label{E:cher}
\mathbb{P}(\,\vert\,Z_n-\mathbb{E}[Z_n]\,\vert\,>
\eta\,\mathbb{E}[Z_n]\,) \le
        2 e^{-c_\eta \mathbb{E}[Z_n]}\, .
\end{equation}
In Lemma 3 we shall use
\begin{equation}\label{E:lde}
\mathbb{P}(\,Z_n\,<
(1-\eta)\,\mathbb{E}[Z_n]\,) \le
       e^{-\frac{\eta^2}{2}\cdot\mathbb{E}[Z_n]}\, .
\end{equation}
In the following we shall assume that $n$ is always sufficiently large.
Let us next recall Chebyshev's inequality \cite{Ross}: suppose $\xi$ is a
r.v.~having finite variance, $\mathbb{V}(\xi)$, and $m>0$. Then
\begin{equation}\label{E:cheby}
\mathbb{P}(|\xi-\mathbb{E}(\xi)|\ge m)\le
\frac{\mathbb{V}(\xi)}{m^2}.
\end{equation}
Furthermore, the r.v.~$X$ is $\mbox{Bi}(n,\lambda_n)$-distributed if
$$
\mathbb{P}(X=\ell)=\binom{n}{\ell}\lambda_n^\ell\,(1-\lambda_n)^{n-\ell}
$$
and we call $X$ binomially distributed (with parameters $n,\lambda_n$).

We next come to some basic facts about binomial branching processes,
$\mathcal{P}_n=\mathcal{P}_n(p)$ \cite{Harris:63,Kolchin:86}.
Suppose the process $\mathcal{P}_n$ is initialized at $\xi$.
Let $(\xi_i^{(t)})$, $i,t\in\mathbb{N}$ count the number of ``offspring''
of the $i$th-individual of generation $(t-1)$ and in particular $\xi_1^{(1)}$
counts the number of offspring generated by $\xi$, in which all the
r.v.s~$\xi_i^{(t)}$ are Bi$(n,p)$-distributed.
Let $\mathcal{P}_0=\mathcal{P}_0(p)$ denote the branching process for
which $\xi_1^{(1)}$ is Bi$(n,p)$- and all $\xi_i^{(t)}\ne\xi_1^{(1)}$ are
Bi$(n-1,p)$-distributed. Furthermore, let $\mathcal{P}_P(\lambda)$,
$(\lambda>0)$ denote the Poisson branching process in which all
individuals $\xi_i^{(t)}$ generate offspring according to the Poisson
distribution, i.e., $\mathbb{P}(\xi_{i}^{(t)}=j)=
\frac{\lambda^j}{j!}e^{-\lambda}$.
We accordingly consider the family of r.v.~$(Z_i^x)_{i\in \mathbb{N}_0}$:
$Z_0^x=1$ and $Z^x_{t} =
\sum_{i=1}^{Z^x_{t-1}}\xi_i^{(t)}$ for $t\ge 1$ and interpret $Z^x_t$ as
the number of individuals ``alive'' in generation $t$, where
$x\in\{n,0,P\}$.
Of particular interest for us will be the limit
$\lim_{t\to\infty}\mathbb{P}(Z_t^x>0)$, i.e.~the probability of
infinite survival. We write
$$
\pi_0(p)=\lim_{t\to\infty}\mathbb{P}(Z_t^0>0), \
\pi_n(p)=\lim_{t\to\infty}\mathbb{P}(Z_t^n>0)
\ \text{\rm and} \
\pi_P(\lambda)=\lim_{t\to\infty}\mathbb{P}(Z_t^P>0)
$$
for the survival probability of $\mathcal{P}_0$, $\mathcal{P}_n$ and
$\mathcal{P}_P(\lambda)$, respectively.

\begin{lemma}\label{C:1}\cite{Bollobas:92}
Let $p=\chi_n/n$ where $\chi_n > 1$, then $\pi_0(p) =
(1 + o(1))\pi_P(\chi_n)$, where $\pi_P(\chi_n)>0$ is the unique
positive root of the equation $e^{-\chi_n y}=1-y$. Particularly, if
$\chi_n=1+\epsilon_n$ where $0<\epsilon_n=o(1)$ and $s=o(n\epsilon_n)$,
\begin{equation*}
\pi_0(p)=(1+ o(1))\pi_{n-s}(p)=(2+o(1))\epsilon_n.
\end{equation*}
\end{lemma}

\begin{proof}
Let $f_{m}(s)$ be the probability generating function for the binomial
distribution ${\rm Bi}(m,\frac{\chi_n}{n})$ and $g_{\chi_n}(s)$
be the probability generating function for Poisson distribution with parameter
$\lambda=\chi_n$, i.e.,
\begin{eqnarray*}
f_m(s)&=&\sum_{j=1}^m P(\xi_i^{(t)}=j)\cdot s^j\\
&=&\sum_{j=1}^m\binom{m}{j}(\frac{\chi_ns}{n})^j(1-\frac{\chi_n}{n})^{m-j}\\
&=&\left[1-(1-s)\frac{\chi_n}{n}\right]^m\\
g_{\chi_n}(s)&=&\sum_{i=0}^{\infty}e^{-\chi_n}\cdot\frac{(\chi_n)^i}{i!}
\cdot s^i=e^{(s-1)\chi_n}.\\
\end{eqnarray*}
Then $\pi_n$ and $\pi_{\chi_n}$, the survival probabilities for the binomial
distribution and Poisson distribution, are the roots of $f_n(1-s)=1-s$ and
$g_{\chi_n}(1-s)=1-s$, respectively. Clearly, $f_n(1-s)=g_{\chi_n}(1-s)
e^{O(\frac{1}{n})}$, whence
\begin{eqnarray}
f_n(1-\pi_{\chi_n}+o(1))&=& g_{\chi_n}(1-\pi_{\chi_n}+o(1))\cdot
e^{O(\frac{1}{n})}\nonumber\\
&=&e^{-\pi_{\chi_n}\chi_n}e^{o(1)\chi_n
+O(\frac{1}{n})}\nonumber \\
&=&e^{-\pi_{\chi_n}\chi_n}(1+o(1))=1-\pi_{\chi_n}+o(1)\label{E:jjj}.
\end{eqnarray}
Since $E(\xi_i^{(t)})=f_n'(1)=\frac{\chi_n}{n}n=\chi_n>1$, where
$\xi_i^{(t)}$ counts the number of ``offspring'' of the $i$th-individual of
generation $(t-1)$, we can conclude that $\pi_n$ is the unique positive
root of $f_n(1-s)=1-s$.
In view of eq.~(\ref{E:jjj}) we have $\pi_n=\pi_{\chi_n}+o(1)
=\pi_{\chi_n}(1+o(1))$. This implies
$$
\pi_0(\frac{\chi_n}{n})=(1+o(1))\pi_n=\pi_{\chi_n}(1+o(1)),
$$
where $x=\pi_{\chi_n}$ is the unique positive root of $e^{-\chi_n\cdot x}=1-x$.
In case of $0<\epsilon_n=o(1)$, we can compute $\pi_n$ explicitly via the
binomial branching process $\mathcal{P}_m(\frac{\chi_n}{n})$.
To this end we consider the root of $f_{n-k}(1-s)=1-s$ where
$k=o(n\epsilon_n)$ and observe
\begin{eqnarray*}
\pi_n(\frac{1+\epsilon_n}{n})&=&\frac{2n\epsilon_n}{n-1}+O(\epsilon_n^2)
=2\epsilon_n+O(\frac{\epsilon_n}{n})+O(\epsilon_n^2)=(2+o(1))\epsilon_n\\
\pi_{n-k}(\frac{1+\epsilon_n}{n})&=&2\epsilon_n+O(\frac{\epsilon_n}{n})+
O(\frac{k}{n})+O(\epsilon_n^2)=(2+o(1))\epsilon_n.
\end{eqnarray*}
Using $\pi_{n-k}(\frac{1+\epsilon_n}{n})\le \pi_0(\frac{1+\epsilon_n}{n})
\le \pi_n(\frac{1+\epsilon_n}{n})$, we arrive at
$$
\pi_0(\frac{1+\epsilon_n}{n})=(1+o(1))\pi_n(\frac{1+\epsilon_n}{n})
=(1+o(1))(2+o(1))\epsilon_n=(2+o(1))\epsilon_n
$$
and the lemma follows.
\end{proof}

\section{Components of polynomial size}
\label{S:3}
Let $\epsilon$ be a positive constant satisfying $0<\epsilon < 1$.
Suppose $y=x>0$ is the unique positive root of $\exp(-(1+\epsilon)y)=1-y$ and
\begin{equation}\label{E:it}
\wp(\epsilon_n)=
\begin{cases}(1+o(1)) x & \text{\rm for
$\epsilon_n=\epsilon>0$} \\
(2+o(1))\epsilon_n & \text{\rm for $0<\epsilon_n=o(1)$}.
\end{cases}
\end{equation}
According to Lemma~\ref{C:1},
$\wp(\epsilon_n)=\pi_0(\frac{1+\epsilon_n}{n-1})$ is the survival
probability of branching process $\mathcal{P}_0(\frac{1+\epsilon_n}{n-1})$.
For $k\in\mathbb{N}$ we set
\begin{equation}\label{E:part}
\mu_n  =  \lfloor \frac{1}{2k(k+1)}n^{\frac{2}{3}}\rfloor,\quad
\ell_n = \lfloor\frac{k}{2(k+1)}
n^{\frac{2}{3}}\rfloor,\quad\text{\rm and}\quad
r_n=n-k\mu_n-\ell_n .
\end{equation}
Without loss of generality we can assume $\mu_n,\ell_n,r_n\in\mathbb{N}$
and establish some basic properties of the Cayley graph $\Gamma(S_n,T_n)$:
\begin{lemma}\label{L:gut}
Let $T_n$ be a minimal generating set of $S_n$ consisting of transpositions,
then we have \\
{\rm (1)} $T_n$ has cardinality $n-1$ and corresponds uniquely to a labeled
tree over $[n]$, denoted by ${\mathcal T}_n$. \\
{\rm (2)} there exists a sequence $(v_i)_{2\le i}$ such that $T_n =
\{(v_i\,s_i)\mid 2\le i\le n\}$ and
\begin{eqnarray}
\label{E:wichtig} \forall \;j<i;\quad
x_{v_i}=((x_1,\dots,x_n)\cdot(v_j\,s_j))_{v_i} &\neq &
((x_1,\dots,x_n)\cdot(v_i\,s_i))_{v_i}.
\end{eqnarray}
{\rm (3)} the diameter of $\Gamma(S_n,T_n)$ is given by
\begin{equation}\label{E:diam}
{\rm diam}(\Gamma(S_n,T_n))\le \binom{n}{2}.
\end{equation}
\end{lemma}
\begin{proof}
It is straightforward to prove by induction that $\vert
T_n\vert=n-1$. We next consider the graph $\mathcal{T}_n$ over
$[n]$, having edge-set $T_n$. Since $\langle T_n\rangle=S_n$,
$\mathcal{T}_n$ is connected and since $T_n$ is independent,
$\mathcal{T}_n$ is a tree. This establishes the mapping
\begin{equation*}
\psi\colon \{T_n\mid \text{\rm $T_n$ is a maximal independent
transposition set}\}\longrightarrow \{\mathcal{T}_n\mid \text{\rm
$\mathcal{T}_n$ is a tree over $[n]$}\}.
\end{equation*}
Furthermore, $\psi$ has an inverse; as the edges of a tree over
$[n]$ give rise to a maximal independent set of transpositions that
generate $S_n$, whence assertion (1).
Note that the critical probability $\lambda_n=\frac{1+\epsilon_n}{n-1}$
of Theorem~\ref{T:main} is determined by the cardinality of the
generator set $T_n$, i.e., $\vert T_n\vert=n-1$.\\
In order to prove (2), we generate the tree $\mathcal{T}_n$
inductively as follows: we start with vertex $1$ by setting
$\mathcal{T}_1=\varnothing$ and $v_1=1$. Given $\mathcal{T}_i$, we
consider the transposition $(v_{i+1}\,s_{i+1})$, where $v_{i+1}$ is
the unique minimal element contained in $\mathcal{T}_n\setminus
\mathcal{T}_i$, having minimal distance to $1$, and $s_{i+1}$ is its
unique $\mathcal{T}_i$-neighbor. We then set
$\mathcal{T}_{i+1}=\mathcal{T}_i\cup \{(v_{i+1}\,s_{i+1})\}$. This
process gives rise to the sequence of trees
$\mathcal{T}_2\subset\mathcal{T}_3\subset \dots\subset
\mathcal{T}_n$ and denoting the vertex sets of $\mathcal{T}_i$ by
$V_i$, we have $V_1=\{1\} \subset V_2\subset V_3\subset \dots
V_{n-1}\subset V_{n}=[n]$ where $\{v_i\}=V_i\setminus V_{i-1}$. By
construction
\begin{eqnarray*}
\forall \, j<i; \quad x_{v_i}=((x_1,\dots,x_n)\cdot(v_j\,s_j))_{v_i} & \neq &
((x_1,\dots,x_n)\cdot(v_i\,s_i))_{v_i},
\end{eqnarray*}
where $(x_1,\dots,x_n)\cdot(v_j\,s_j)$ is the product of permutations and
$((\tilde{x}_1,\dots,\tilde{x}_n))_{v_i}=\tilde{x}_{v_i}$. In other words,
we order the $T_n$-transpositions via the sequence of trees
$\{\mathcal{T}_i\}$, such that the transpositions added before
$(v_{i}\,s_{i})$ will not transpose the element $x_{v_{i}}$.
To prove (3) we can, without loss of generality, restrict ourselves to
the case where we have an arbitrary permutation $(x_1,\dots,x_n)$ and
$(y_1,\dots,y_n)$, the unique permutation satisfying $y_{v_i}=i$.
We proceed by constructing a $\Gamma(S_n,T_n)$-path between these two
permutations. Obviously, there exists a unique $v_j$ such that
$n=x_{v_j}$ and in the tree $\mathcal{T}_n$ there exists a unique path
of length
at most ${\rm diam}(\mathcal{T}_n)\le n-1$ connecting $v_j$ and $v_n$.
Accordingly, there is a $\Gamma(S_n,T_n)$-path of length at most
${\rm diam}(\mathcal{T}_n)$ between $(x_i)$ and a permutation $(z_i)$
such that $z_{v_n}=n$. Our construction in (2) implies
\begin{equation*}
\forall \, i<n;\quad ((z_1,\dots,z_n)\cdot(v_i\,s_i))_{v_n}=n,
\end{equation*}
whence we can proceed inductively, moving $(n-1)$ to the $v_{n-1}$th
position using the subtree $\mathcal{T}_{n-1}$. We consequently arrive at
\begin{equation*}
{\rm diam}(\Gamma(S_n,T_n))\le  \sum_{i=2}^{n}{\rm diam}(\mathcal{T}_i)\le
\binom{n}{2}
\end{equation*}
and the proof of the lemma is complete.
\end{proof}

In case of star-transpositions, i.e.~$T_n=P_n=\{(1\,j)\mid 2\le j\le
n\}$, we have the following situation:
\begin{equation}
\{1\}\subset \{(1\,2)\}\subset \{(1\,2),(1\,3)\}\subset \dots\subset
\{(1\,j)\mid 2\le j\le n\},
\end{equation}
$(v_i\,s_i)=(i\,1)$ i.e.~$s_i=1$ and ${\rm diam}(\Gamma(S_n,P_n))=
\lfloor\frac{3(n-1)}{2}\rfloor$, which can be derived from a theorem
of Pak \cite{Pak:99}, being strictly less than $\binom{n}{2}$.

{\bf Example $1$.}
Consider the Cayley graph $\Gamma(S_5,P_5)$ and generate the trees
$\{\mathcal{T}_i\}_{i=1}^5$ inductively. Setting $\mathcal{T}_1=\varnothing$
and $v_1=1$ we select the minimal element in distance $1$ to $v_1$ and
set $v_2=2$, $\mathcal{T}_2=\{(1\,2)\}$. We proceed by selecting
the minimal element in distance $1$ to the vertex set $\{1,2\}$ and
set $v_3=3$, $\mathcal{T}_3=\{(1\,2),(1\,3)\}$.
Finally, we select the minimal element in distance $1$ to the vertex set
$\{1,2,3\}$ and set $v_4=4$, $\mathcal{T}_4=\{(1\,2),(1\,3),(1\,4)\}$.
The only remaining vertex $v_5=5$ is the minimal element in distance $1$ to
the vertex set $\{1,2,3,4\}$ and $\mathcal{T}_5=\{(1\,2),(1\,3),(1\,4),
(1\,5)\}$.
\begin{center}
\setlength{\unitlength}{4pt}
\begin{picture}(100,10)(-5,0)
\put(0,10){\line(-1,-1){5}}
\put(0,10){\circle*{0.5}}\put(1,11){\small{$1$}}
\put(-5,5){\circle*{0.5}}\put(-6,2){\small{$2$}}
\put(0,10){\line(-1,-2){2.5}}
\put(-2.5,5){\circle*{0.5}}\put(-3,2){\small{$3$}}
\put(0,10){\line(1,-2){2.5}}
\put(2.5,5){\circle*{0.5}}\put(2,2){\small{$4$}}
\put(0,10){\line(1,-1){5}}\put(5,5){\circle*{0.5}}
\put(5,2){\small{$5$}}\put(6,7.5){is generated via}
\put(30,10){\line(-1,-1){5}}
\put(30,10){\circle*{0.5}}\put(31,11){\small{$1$}}
\put(25,5){\circle*{0.5}}\put(25,5){\circle{2}}
\put(24,2){\small{$2$}}
\put(33,7.5){\vector(1,0){5}}
\put(24,-1){\small{$v_2=2$}}
\put(45,10){\line(-1,-1){5}}
\put(45,10){\circle*{0.5}}\put(46,11){\small{$1$}}
\put(40,5){\circle*{0.5}}\put(39,2){\small{$2$}}
\put(45,10){\line(-1,-2){2.5}}
\put(42.5,5){\circle*{0.5}}\put(42,2){\small{$3$}}
\put(42.5,5){\circle{2}}
\put(48,7.5){\vector(1,0){5}}
\put(39,-1){\small{$v_3=3$}}
\put(60,10){\line(-1,-1){5}}
\put(60,10){\circle*{0.5}}\put(61,11){\small{$1$}}
\put(55,5){\circle*{0.5}}\put(54,2){\small{$2$}}
\put(60,10){\line(-1,-2){2.5}}
\put(57.5,5){\circle*{0.5}}\put(57,2){\small{$3$}}
\put(60,10){\line(1,-2){2.5}}
\put(62.5,5){\circle*{0.5}}\put(62,2){\small{$4$}}
\put(62.5,5){\circle{2}}
\put(65,7.5){\vector(1,0){5}}
\put(55,-1){\small{$v_4=4$}}
\put(77,10){\line(-1,-1){5}}
\put(77,10){\circle*{0.5}}\put(78,11){\small{$1$}}
\put(72,5){\circle*{0.5}}\put(71,2){\small{$2$}}
\put(77,10){\line(-1,-2){2.5}}
\put(74.5,5){\circle*{0.5}}\put(74,2){\small{$3$}}
\put(77,10){\line(1,-2){2.5}}
\put(79.5,5){\circle*{0.5}}\put(79,2){\small{$4$}}
\put(77,10){\line(1,-1){5}}\put(82,5){\circle*{0.5}}
\put(82,2){\small{$5$}}\put(82,5){\circle{2}}
\put(73,-1){\small{$v_5=5$}}
\end{picture}
\end{center}
Lemma~\ref{L:gut} provides the upper bound $\sum_{i=2}^5{\rm diam}(
\mathcal{T}_i)=7$, where ${\rm diam}(\Gamma(S_5,P_5))=6$ and the
distance between $id=(1,2,3,4,5)$ and $(1,3,2,5,4)$ is the diameter
of $\Gamma(S_5,P_5)$.

We next discuss the bubble-sort graph, $T_n=B_n=\{(i\,i+1)\mid 1\le i\le n-1\}$.
In view of
\begin{equation}
\{1\}\subset \{(1\,2)\}\subset \{(1\,2),(2\,3)\}\subset \dots\subset
\{(i\,i+1)\mid 1\le i\le n-1\}
\end{equation}
we arrive at $(v_i\,s_i)=(i\,i-1)$ and ${\rm diam}(\Gamma(S_n,B_n))=
\binom{n}{2}$.

{\bf Example $2$.}
In order to make the above explicit we consider the Cayley graph
$\Gamma(S_5,B_5)$ and generate the trees $\{\mathcal{T}_i\}_{i=1}^5$
inductively. Setting $\mathcal{T}_1=\varnothing$ and $v_1=1$,
we select the minimal element in distance $1$ to $v_1$ and set
$v_2=2$, $\mathcal{T}_2=\{(1\,2)\}$. We proceed by selecting the
minimal element in distance $1$ to the vertex set $\{1,2\}$ and set
$v_3=3$, $\mathcal{T}_3=\{(1\,2),(2\,3)\}$. Finally we select the
minimal element in distance $1$ to the vertex set $\{1,2,3\}$ and
set $v_4=4$, $\mathcal{T}_4=\{(1\,2),(2\,3),(3\,4)\}$. Then $v_5=5$
is the minimal element in distance $1$ to the vertex set $\{1,2,3,4\}$ and
$\mathcal{T}_5=\{(1\,2),(2\,3),(3\,4),(4\,5)\}$.
\begin{center}
\setlength{\unitlength}{4pt}
\begin{picture}(100,20)(-10,-5)
\put(0,10){\line(0,1){5}}\put(0,10){\circle*{0.5}}
\put(0,15){\circle*{0.5}}\put(-2,9){\small{$2$}}
\put(-2,14){\small{$1$}}\put(0,10){\line(0,-1){5}}
\put(0,5){\circle*{0.5}}\put(-2,4){\small{$3$}}
\put(0,5){\line(0,-1){5}}
\put(0,0){\circle*{0.5}}\put(-2,-1){\small{$4$}}
\put(0,0){\line(0,-1){5}}
\put(0,-5){\circle*{0.5}}\put(-2,-6){\small{$5$}}
\put(4,6){is generated via}
\put(27,5){\line(0,1){5}}\put(27,5){\circle*{0.5}}
\put(27,10){\circle*{0.5}}\put(24,4){\small{$2$}}
\put(24,9){\small{$1$}}\put(27,5){\circle{2}}
\put(30,7){\vector(1,0){5}}
\put(22,0){$v_2=2$}
\put(40,10){\line(0,1){5}}\put(40,10){\circle*{0.5}}
\put(40,15){\circle*{0.5}}\put(37,9){\small{$2$}}
\put(37,14){\small{$1$}}\put(40,10){\line(0,-1){5}}
\put(40,5){\circle*{0.5}}\put(37,4){\small{$3$}}
\put(40,5){\circle{2}}
\put(43,7){\vector(1,0){5}}
\put(35,0){$v_3=3$}
\put(53,10){\line(0,1){5}}\put(53,10){\circle*{0.5}}
\put(53,15){\circle*{0.5}}\put(50,9){\small{$2$}}
\put(50,14){\small{$1$}}\put(53,10){\line(0,-1){5}}
\put(53,5){\circle*{0.5}}\put(50,4){\small{$3$}}
\put(53,5){\line(0,-1){5}}
\put(53,0){\circle*{0.5}}\put(50,-1){\small{$4$}}
\put(53,0){\circle{2}}
\put(56,7){\vector(1,0){5}}
\put(48,-5){$v_4=4$}
\put(66,10){\line(0,1){5}}\put(66,10){\circle*{0.5}}
\put(66,15){\circle*{0.5}}\put(63,9){\small{$2$}}
\put(63,14){\small{$1$}}\put(66,10){\line(0,-1){5}}
\put(66,5){\circle*{0.5}}\put(63,4){\small{$3$}}
\put(66,5){\line(0,-1){5}}
\put(66,0){\circle*{0.5}}\put(63,-1){\small{$4$}}
\put(66,0){\line(0,-1){5}}
\put(66,-5){\circle*{0.5}}\put(63,-6){\small{$5$}}
\put(66,-5){\circle{2}}
\put(63,-10){$v_5=5$}
\end{picture}
\end{center}
Lemma~\ref{L:gut} provides the upper bound $\sum_{i=2}^5{\rm diam}(
\mathcal{T}_i)=10$, and ${\rm diam}(\Gamma(S_5,B_5))=10$. The
distance between $id=(1,2,3,4,5)$ and $(5,4,3,2,1)$ is the diameter
of $\Gamma(S_5,B_5)$.

\begin{lemma}\label{L:minigene}
Suppose $T_n$ is a minimal generating set of transpositions. We select
permutations with independent probability
$\lambda_n=\frac{1+\epsilon_n}{n-1}$,
where $n^{-\frac{1}{3}+\delta}\le \epsilon_n$, for some $\delta>0$. Then each
permutation, $v$, is contained in a $\Gamma_n$-subtree $\mathcal{T}_n(v)$
of size $\lfloor\frac{1}{4}n^{\frac{2}{3}}\rfloor$ with probability at least
$\wp(\epsilon_n)$.
\end{lemma}
\begin{proof}
We construct the subtree $\mathcal{T}_n(v)$ by means of a branching
process \cite{Harris:63} within $\Gamma(S_n,T_n)$. Without loss of
generality, we may initiate the process at $id$ and have
$r_n=n-\frac{1}{2}n^{\frac{2}{3}}\in\mathbb{N}$.
We shall begin by specifying an appropriate move-set
(of transpositions) by which
the offspring of the branching process is being generated. To this end, let
\begin{equation*}
N=\{(v_j\,s_j)\mid 1\le j\le n-\frac{1}{2}n^{\frac{2}{3}}-1\}\subset
T_n.
\end{equation*}
Note that $N$ acts trivially on labels $v_h$ where
$h>n-\frac{1}{2}n^{\frac{2}{3}}-1$.\\
The process is defined as follows: we set $U_0=\varnothing\subset N$
and $M_0=L_0=\{id\}\subset S_n$. At step $(j+1)$,
suppose we are given $U_j\subset N$, $M_j$ and $L_j\subset S_n$.
In case of $L_j=\varnothing$ or $\vert
U_j\vert=\lfloor\frac{1}{4}n^{\frac{2}{3}}\rfloor-1$ the process
stops.
Otherwise, we consider the smallest element $l_j\in L_j$ and select
among its smallest $(n-\lfloor\frac{3}{4}n^{2/3}\rfloor-1)$
neighbors, contained in $N\setminus U_j$ with independent
probability $\lambda_n$.
Let $x_1=l_j\,r_{x_1}$ be the first selected $l_j$-neighbor and
$r_{x_1}\in N\setminus U_j$. We then set $U_j(x_1)=U_j\dot\cup
\{r_{x_1}\}$ and proceed the selection with the smallest
$(n-\lfloor\frac{3}{4}n^{2/3}\rfloor-1)$ neighbors contained in
$N\setminus U_j(x_1)$ instead of those in $N\setminus U_j$. After
all $l_j$ neighbors are checked and given that $(x_1,\dots,x_s)$
have been subsequently selected, we set
\begin{eqnarray*}
U_{j+1} & = & U_j\dot\cup\{r_{x_1},\dots,r_{x_s}\} \\
L_{j+1} & = & (L_j\setminus \{l_j\})\cup\{x_1,\dots,x_s\}\\
M_{j+1} & = & M_j\dot  \cup\{ x_1,\dots,x_s\}.
\end{eqnarray*}
The minimality of $T_n$ and the fact that each $T_n$-element is used
at most once implies that this process generates a tree, i.e.~each
$M_{j+1}$-element is considered only once. Furthermore, in view of
\begin{equation}
\frac{1+\epsilon_n}{n-1}\cdot
\left(n-\lfloor\frac{3}{4}n^{\frac{2}{3}}\rfloor-1\right)>1.
\end{equation}
Relating our construction with the binomial branching process
$\mathcal{P}_m(\frac{1+\epsilon_n}{n-1})$, where $m=n-\lfloor\frac{3}{4}n^{
\frac{2}{3}}\rfloor-1$, we observe
\begin{eqnarray*}
\mathbb{P}\left(\vert
M_j\vert=\lfloor\frac{1}{4}n^{\frac{2}{3}}\rfloor\mid \text{\rm for
some $j$} \right)\ge
\pi_{m}\left(\frac{1+\epsilon_n}{n-1}\right)=\wp(\epsilon_n).
\end{eqnarray*}
Indeed, the above equation holds for
$\epsilon_n\ge n^{-\frac{1}{3}+\delta}$.
In case of $0<\epsilon_n=o(1)$ we notice
$\lfloor\frac{3}{4}n^{\frac{2}{3}}\rfloor=o(n\cdot\epsilon_n)$. Therefore
Lemma~\ref{C:1}, (2) implies
$\pi_m(\frac{1+\epsilon_n}{n-1})=(2+o(1))\epsilon_n=\wp(\epsilon_n)$.
In case of $0<\epsilon_n=\epsilon<1$, we consider the probability generating
functions for both: the binomial distribution,
$\mathcal{P}_m(\frac{1+\epsilon}{n-1})$ and the Poisson distribution,
$\mathcal{P}_P(1+\epsilon)$.
Let $f_{n-1}(s)$ be the probability generating function for the binomial
distribution ${\rm Bi}(n-1,\frac{1+\epsilon}{n-1})$ and $g_{1+\epsilon}(s)$
be the probability generating function for Poisson distribution
with parameter $\lambda=1+\epsilon$, i.e.
\begin{eqnarray*}
f_{n-1}(s) & = & \sum_{j=0}^{n-1} P(\xi_i^{(t)}=j)\cdot s^j \\
           & = & \sum_{j=1}^{n-1}\binom{n-1}{j}
               \left(\frac{1+\epsilon}{n-1}\right)^j
\left(1-\frac{1+\epsilon}{n-1}\right)^{n-j}s^j\\
          & = & \left[1-(1-s)\frac{1+\epsilon}{n-1}\right]^{n-1}\\
g_{1+\epsilon}(s) & = & \sum_{i=0}^{\infty}e^{-(1+\epsilon)}\cdot
                \frac{(1+\epsilon)^i}{i!}\cdot s^i=e^{(s-1)(1+\epsilon)}.
\end{eqnarray*}
Clearly, $f_{n-1}(1-s)=g_{1+\epsilon}(1-s)e^{O(\frac{1}{n-1})}$ and
$f_{m}(1-s)=f_{n-1}(1-s)\cdot (1-s\frac{1+\epsilon}{n-1})^{-\lfloor\frac{3}{4}
n^{\frac{2}{3}}\rfloor}$.
By studying the roots of $f_{m}(1-s)=1-s$, $f_{n-1}(1-s)=1-s$ and
$g_{1+\epsilon}(1-s)=1-s$, we derive
$$
\pi_m\left(\frac{1+\epsilon}{n-1}\right)=(1+o(1))\pi_{n-1}\left(\frac{1+\epsilon}{n-1}
\right)=(1+o(1))\pi_{P}(1+\epsilon)=\wp(\epsilon)
$$
and the lemma follows.
\end{proof}
For given $\delta$, by choosing $k$ sufficiently large, we proceed by
enlarging the trees of Lemma~\ref{L:minigene} to subcomponents of arbitrary
polynomial size.
We remark that Lemma~\ref{L:gut} is of central importance for the
construction of the subcomponents of Lemma~\ref{L:expand2}.

\begin{lemma}\label{L:expand2}
Given $k\ge 2$ and $\delta>0$, $\lambda_n=\frac{1+\epsilon_n}{n-1}$,
where $n^{-\frac{1}{3}+\delta}\le \epsilon_n$,
there exists a function $\theta_{n,k}$, with
the property $\theta_{n,k}\ge \frac{1}{4k(k+1)}n^{\delta}$.
Then each $\Gamma_n$-vertex is contained in a $\Gamma_n$-subcomponent of
size at least
$$
\frac{1}{2^{k+2}}
\cdot\left[\frac{1}{4k(k+1)}\right]^k\cdot n^{\frac{2}{3}+k\delta}
$$
with probability at least
\begin{equation}
\delta_k(\epsilon_n)=\wp(\epsilon_n)\, (1-e^{-\beta_{k,n}\theta_{n,k}}),
\end{equation}
where $0<\beta_{k,n}<1$ and $\epsilon_n\ge n^{-\frac{1}{3}+\delta}$.
\end{lemma}
\begin{proof}
Without loss of generality we may assume $\pi=id$, $\mu_n\in\mathbb{N}$
and set for all $1\le m\le k$,
\begin{equation*}
A_{m}=\left\{(v_j^m\,s_j^m)\in T_n \mid 1\le j\le\mu_n\right\}.
\end{equation*}
where
$(v_j^m\,s_j^m)=(v_{r_n+j+(m-1)\mu_n-1}\,s_{r_n+j+(m-1)\mu_n-1})$ and
$r_n=n-\lfloor\frac{1}{2}n^{\frac{2}{3}}\rfloor$, see eq.~(\ref{E:part}).
That is, $A_{m}$ is the ``first'' (in the sense of the labeling
given by the sequence $(v_{r_n},v_{r_n+1},\dots,v_n)$) subset of
$T_n$-transpositions
that act on labels $v_i$, where $i\le r_n+m\mu_n-1$ for $1\le m\le k$.
Furthermore,  for $1\le m\le k$, $|A_m|=\mu_n=\lfloor \frac{1}{2k(k+1)}
n^{\frac{2}{3}}\rfloor$, see eq.~(\ref{E:part}).
We set $w_{j}^{(h)}=(v_j^h\,s_j^h)\in A_h$ and consider the branching
process of Lemma~\ref{L:minigene} at $\pi=id$, assuming that we obtain a tree
$T^1$ of size $\lfloor \frac{1}{4}n^{\frac{2}{3}}\rfloor$. Let
\begin{eqnarray*}
Y_1=\left|\{w_{i}^{(1)}\in A_1\mid \exists x\in T^1;x \cdot w_{i}^{(1)}
\in \Gamma_n\}\right|.
\end{eqnarray*}
According to Lemma~\ref{L:gut}
\begin{equation*}
\forall\, x,y\in T^1;\forall\, w_{i}^{(1)}\neq w_{r}^{(1)}\in A_1;\quad
x\cdot w_{i}^{(1)}\neq  y\cdot w_{r}^{(1)},
\end{equation*}
whence
\begin{eqnarray}\label{E:exp1Y1}
\mathbb{E}[Y_1]=\mu_n\cdot
\left(1-\left(1-\frac{1+\epsilon_n}{n-1}\right)^{
\frac{1}{4}n^{\frac{2}{3}}}\right) \sim
\mu_n\left(1-\exp(-(1+\epsilon_n)
\frac{1}{4}n^{-\frac{1}{3}})\right).
\end{eqnarray}
Using large deviation inequalities eq.~(\ref{E:lde}) \cite{Chernoff:52},
we conclude that $\beta_1=\frac{1}{8}>0$ satisfies
\begin{eqnarray*}
\mathbb{P}\left(Y_1<\frac{1}{2}\mathbb{E}[Y_1]\right)\le
\exp\left(-\beta_1\cdot \mathbb{E}[Y_1]\right).
\end{eqnarray*}
We select the smallest element, $x_{(i\,j)}$, from the set $\{x
\cdot w_{j}^{(1)}  \mid x\in T^1,x \cdot w_{j}^{(1)}\in\Gamma_n\}$
and start the branching process of Lemma~\ref{L:minigene} at
$x_{(i\,j)}$. As a result, we derive the tree $C_2(x_{(i\,j)})$ of
size $\lfloor \frac{1}{4}n^{\frac{2}{3}}\rfloor$ with probability at
least $\wp(\epsilon_n)$. However, note that $T^1 \cup
C_2(x_{(i\,j)})$ may not be tree any more. According to
Lemma~\ref{L:minigene}, the generation of this tree
$C_2(x_{(i\,j)})$ exclusively involves labels $v_j$ where $j\le
r_n-1$. Therefore, since any two smallest elements $x_{(i_1\,j_1)}$
and $x_{(i_2\,j_2)}$ differ in at least one of two coordinates with
labels $v_{j_1},v_{j_2}$ for $r_n\le j_1,j_2\le r_n+\mu_n$, we have
\begin{equation*}
C_2(x_{(i_1\,j_1)})\cap C_2(x_{(i_2\,j_2)})=\varnothing.
\end{equation*}
Let $X_1$ be the r.v.~counting the number of these new
$\Gamma_n$-subcomponents. In view of eq.~(\ref{E:exp1Y1}), we obtain
\begin{eqnarray*}
\mathbb{E}[X_1]=\wp(\epsilon_n)\cdot
\mathbb{E}[Y_1]\sim \wp(\epsilon_n)\cdot\mu_n
\left(1-\exp(-(1+\epsilon_n)\frac{1}{4}n^{-\frac{1}{3}})\right).
\end{eqnarray*}
In order to make the dependence of
$\theta_{n,k}=\wp(\epsilon_n)\cdot\mu_n\left(1-\exp(-(1+\epsilon_n)
\frac{1}{4}n^{-\frac{1}{3}})\right)$ for fixed $\delta>0$ on $k$ and $n$
explicit, we compute
\begin{eqnarray*}
\theta_{n,k} &\ge &
2\cdot n^{-\frac{1}{3}+\delta}\cdot \frac{1}{2k(k+1)}n^{\frac{2}{3}}\cdot
(1+n^{-\frac{1}{3}+\delta})\cdot\frac{1}{4}\cdot n^{-\frac{1}{3}}-o(1)\\
& = &\frac{1}{4k(k+1)}\cdot n^{\delta} \quad\mbox{ as } n\rightarrow \infty.
\end{eqnarray*}

Again, using large deviation inequalities eq.~(\ref{E:lde}), we
conclude that $\beta_{1}=\frac{1}{8}>0$ satisfies
\begin{equation*}
\mathbb{P}(X_1<\frac{1}{2}\theta_{n,k})\le \exp(-\beta_{1}\theta_{n,k})
\end{equation*}
or equivalently, since the union of all the
$C_2(x_{(i\,j)})$-subcomponents with $T^1$ forms a
$\Gamma(S_n,T_n)$-subcomponent, $T^2$, we have
\begin{equation}
\mathbb{P}\left(\vert T^2\vert <
\lfloor \frac{1}{4}n^{2/3}\rfloor \cdot \frac{1}{2}\theta_{n,k}\right)
\le \exp(-\beta_{1}\theta_{n,k}).
\end{equation}
We now proceed by induction:\\
{\it Claim:} For each $2\le i\le k$, there exists some constant
$\beta_{i,n}>0$ and a $\Gamma(S_n,T_n)$-subcomponent $T^i$ such that
\begin{eqnarray*}
\mathbb{P}(\vert T^i\vert <\lfloor \frac{1}{4}n^{2/3}\rfloor \cdot
\left({\frac{\theta_{n,k}}{2}}\right)^{i-1})\le
\exp(-\beta_{i-1,n}\theta_{n,k}).
\end{eqnarray*}
We have already established the induction basis. As for the
induction step, let us assume the claim holds for $i<k$ and
let $C_i(\alpha)$ denote a subcomponent generated by the branching process
of Lemma~\ref{L:minigene} in the $i$-th step.
We consider the $T_n$-transpositions $w_{r}^{(i+1)}\ne w_{a}^{(i+1)}\in
A_{i+1}$. We consider the minimal elements, $x^\alpha_{r}$ of
\begin{eqnarray*}
Y_{i+1}=\{w_{r}^{(i+1)}\in A_{i+1}\mid \exists\, x\in C_i(\alpha); x\cdot
w_{r}^{(i+1)}\in \Gamma_n\}
\end{eqnarray*}
at which we initiate the branching process of
Lemma~\ref{L:minigene}. The process generates subcomponents
$C_{i+1}(x^\alpha_{r})$ of size $\lfloor
\frac{1}{4}n^{\frac{2}{3}}\rfloor$ with probability $\ge
\wp(\epsilon_n)$. Any two of these are mutually disjoint and let
$X_{i+1}$ be the r.v.~counting their number. We derive setting
$q_n=\lfloor \frac{1}{4}n^{2/3}\rfloor$. In order to make the dependence of
$\beta_{i,n}$ for fixed $\delta>0, k\ge 2$ on $n$ and $i$ explicit, we
set $\beta_{1,n}=\beta_1=\frac{1}{8}$ and recursively define
$\beta_{i,n}$ for $i\ge 2$,
$$
\beta_{i,n}=\beta_{i-1,n}-
\frac{\ln(1+\exp(-\beta_1\theta_{n,k}^{i-1}+\beta_{i-1,n}
\theta_{n,k}))}{\theta_{n,k}}=\beta_{i-1,n}+o(1)\,\,\,\mbox{ for }
k\ge i\ge 2
$$
We compute
\begin{eqnarray*}
\mathbb{P}\left(\vert T^{i+1}\vert < q_n
\frac{1}{2^{i}}\theta_{n,k}^{i}\right) & \le &
\underbrace{\mathbb{P}\left(
\vert T^i\vert <q_n \frac{1}{2^{i-1}}\theta_{n,k}^{i-1}\right)
}_{\text{\rm failure at step $i$}}
\ + \\
&& \underbrace{ \mathbb{P}\left(\vert T^{i+1}\vert <
q_n\frac{1}{2^{i}}\theta_{n,k}^{i}\mbox{ and }
\vert T^i\vert \ge  q_n\frac{1}{2^{i-1}}\theta_{n,k}^{i-1} \right)
}_{\text{\rm failure at step $i+1$ conditional to
$\vert T^i\vert \ge  q_n\frac{1}{2^{i-1}}\theta_{n,k}^{i-1}$}} \\
 & \le & \underbrace{e^{-\beta_{i-1,n}\,\theta_{n,k}}}_{\text{\rm induction
hypothesis}} +
\underbrace{e^{-\beta_1 \,\theta_{n,k}^{i}}}_{\text{\rm large deviation
results}}\cdot
(1-e^{-\beta_{i-1,n}\, \theta_{n,k}})\, , \\
& \le & e^{-\beta_{i,n}\,\theta_{n,k}}
\end{eqnarray*}
and the Claim follows.\\
Therefore, each $\Gamma_n$-vertex is contained in a subcomponent of
size
\begin{equation*}
\ge \frac{1}{4}\cdot n^{\frac{2}{3}}\cdot\frac{1}{2^k}\cdot
\left[\frac{1}{4k(k+1)}\right]^k\cdot n^{k\delta}=\frac{1}{2^{k+2}}
\cdot\left[\frac{1}{4k(k+1)}\right]^k\cdot n^{\frac{2}{3}+k\delta},
\end{equation*}
with probability at least
$\wp(\epsilon_n)(1-e^{-\beta_{k,n}\theta_{n,k}})$ and the lemma is proved.
\end{proof}

\section{Vertices in small components}\label{S:4}
For given $0<\delta<1$, let
\begin{equation}\label{E:polysize}
M_{k}(n)=
\frac{1}{2^{k+2}}\left[\frac{1}{4k(k+1)}\right]^k n^{\frac{2}{3}+k\delta}.
\end{equation}
Let $\Gamma_{n,k}$ denote the set of $\Gamma_n$-vertices contained in
components of size $\ge M_k(n)$ for fixed $0<\delta<1$.
In this section we prove that $\vert \Gamma_{n,k}\vert$ is a.s.~$
\sim \wp(\epsilon_n) \frac{1+\epsilon_n}{n-1} n!$.
In analogy to Lemma~$3$ of \cite{Reidys:08rand} we first observe that the
number of vertices, contained in $\Gamma_n$-components of size
$<M_k(n)$, is sharply concentrated.
The concentration reduces the problem to a computation of expectation
values. It follows from considering the indicator r.vs. of
pairs $(C,v)$ where $C$ is a component and $v\in C$ and to estimate
their correlation. Since the components in question are small,
no ``critical'' correlation terms arise.

Let $U_n=U_n(a)$ denote the set of vertices contained in
components of size $< n^a$ where $a>0$.
Then following the arguments in \cite{Bollobas:92}
\begin{lemma}\label{L:comple1}
Let $a>0$ be a fixed constant. We are given $\delta>0$ and $\lambda_n=\frac{1+\epsilon_n}{n-1}$, where
$1>\epsilon_n\ge n^{-\frac{1}{3}+\delta}$. Then
\begin{equation}
\mathbb{P}\left(\vert\,\vert U_n \vert -\mathbb{E}[\vert
U_n\vert]\,\vert \ge \frac{1}{n} \mathbb{E}[\vert
U_n\vert]\right)=o(1).
\end{equation}
\end{lemma}
\begin{proof}
Let $I_{C,v}$, be the indicator r.v.~of the pair $(C,v)$, where $v\in C$ and
$C\in U_n$ is a component of size
$<n^a$. We have
\begin{eqnarray*}
|U_n|=\sum_{(C,v)}I_{C,v}.
\end{eqnarray*}
and we proceed by proving that the r.v. $|U_n|$ is sharply
concentrated by analyzing the correlation terms
$\mathbb{E}(I_{C_1,v}I_{C_2,w})$.
Correlation may arise in two ways: the pairs
$(C_1,v)$ and $(C_2,w)$ either satisfy $C_1=C_2$ or the minimal
distance, $d_{\Gamma(S_n,T_n)}(C_1,C_2)=2$.
Suppose first $C_1=C_2$, then
\begin{eqnarray*}
\sum_{(C,v)\sim(C,w)}\mathbb{E}(I_{C,v}I_{C,w})
&=&\sum_{(C,v)}\sum_{(C,w)\sim(C,v)}\mathbb{E}(I_{C,v})\\
&\le & \sum_{(C,v)}n^a \mathbb{E}(I_{C,v})=n^a
\mathbb{E}[|U_n|]
\end{eqnarray*}
Secondly we consider the case $C_1\ne C_2$.
Then there exist vertices $v\in C_1$ and $w\in C_2$ with
$d_{\Gamma(S_n,T_n)}(v,w)=2$, i.e.~we have an additional vertex
$u\not\in \Gamma_n$ which, if selected, would lead to a merger
of the subcomponents $C_1$ and $C_2$. Accordingly,
\begin{eqnarray*}
\mathbb{P}(\text{\rm $d(C_1,C_2)=2$}) & = &
\frac{(1-\lambda_n)}{\lambda_n}\ \mathbb{P}(C_1\cup C_2\cup \{u\}
\ \text{\rm is a
$\Gamma_n$-component}) \\
& \le & n \ \mathbb{P}(C_1\cup C_2\cup \{u\} \
\text{\rm is a $\Gamma_n$-component})
\end{eqnarray*}
and we derive, summing over all possible $v,w,u$, the upper bound
\begin{equation*}
\sum_{d(C_1,C_2)=2}\mathbb{E}[I_{C_1,v_1}\,I_{C_2,v_2}]\le
n\, (2n^a+1)^3 \, \vert\Gamma_n\vert  .
\end{equation*}
The uncorrelated pairs $(I_{C_1,v_1},I_{C_2,v_2})$ can be estimated
by
\begin{equation*}
\sum_{(C_1,v_1)\not\sim
(C_2,v_2)}\mathbb{E}[I_{C_1,v_1}\,I_{C_2,v_2}]=\sum_{(C_1,v_1)\not\sim
(C_2,v_2)}\mathbb{E}[I_{C_1,v_1}]\cdot \mathbb{E}[I_{C_2,v_2}]\le
\mathbb{E}[|U_n|]^2 .
\end{equation*}
Consequently we arrive at
\begin{eqnarray*}
\mathbb{E}[U_n(U_n -1)] & = &
\sum_{\substack{(C,v_1)\\\sim
(C,v_2)}}\mathbb{E}[I_{C,v_1}\,I_{C,v_2}]+
\sum_{\substack{(C_1,v_1)\\\sim
(C_2,v_2)}}\mathbb{E}[I_{C_1,v_1}\,I_{C_2,v_2}] +
\sum_{\substack{(C_1,v_1)\\\not\sim
(C_2,v_2)}}\mathbb{E}[I_{C_1,v_1}\,I_{C_2,v_2}]\\
& \le & n^a\; \mathbb{E}[|U_n|] + n\,
(2n^a+1)^3 \vert\Gamma_n\vert + \mathbb{E}[|U_n|]^2 .
\end{eqnarray*}
Just considering isolated vertices implies $\mathbb{E}[U_n]\ge
c\,\vert \Gamma_n\vert$ for some $c>0$, i.e.~the expected number
of vertices in small components grows faster than any polynomial.
Employing Chebyshev's inequality, eq.~(\ref{E:cheby}), we derive
\begin{eqnarray*}
\mathbb{P}\left(\vert |U_n| -\mathbb{E}[|U_n|]\vert
\ge \frac{1}{n}\, \mathbb{E}[|U_n|] \right)
&\le& n^2\, \frac{\mathbb{V}[
|U_n|]}{\mathbb{E}[|U_n|]^2}\\
&=& n^2\frac{\mathbb{E}[|U_n|(|U_n|-1)]+
\mathbb{E}[|U_n|]-\mathbb{E}[|U_n|]^2}
{\mathbb{E}[|U_n|]^2}\\
&\le& n^2\frac{ n^a+
\frac{1}{c}\,n\, (2n^a+1)^3 +1}
{\mathbb{E}[|U_n|]}= o\left(\frac{1}{n^2}\right),
\end{eqnarray*}
whence the lemma.
\end{proof}

With the help of Lemma~\ref{L:comple1}, we proceed by computing the
size of $\Gamma_{n,k}$.
\begin{lemma}\label{L:size1}
Suppose $k\in\mathbb{N}$ is arbitrary but fixed and we are given $\delta>0$.
Let $\omega_n=|\Gamma_n\backslash\Gamma_{n,k}|$ and $\lambda_n=
\frac{1+\epsilon_n}{n-1}$, where $n^{-\frac{1}{3}+\delta}\le\epsilon_n <1$.
Then
\begin{equation}
\vert \Gamma_{n,k}\vert \sim \wp(\epsilon_n)
\frac{1+\epsilon_n}{n-1} n! \ \qquad \text{\it a.s.~.}
\end{equation}
\end{lemma}
\begin{proof}
First we prove for any $n^{-\frac{1}{3}+\delta}\le \epsilon_n\le \lambda$,
where
$\lambda>0$
\begin{equation}
(1-o(1))\wp(\epsilon_n)\,\vert \Gamma_n\vert\le
\vert \Gamma_{n,k}\vert
\qquad \text{\rm  a.s.}
\end{equation}
By Lemma~\ref{L:expand2} we have
$$
\mathbb{E}[\omega_n]\le (1-\delta_{k}(\epsilon_n))
\vert\Gamma_n\vert.
$$
In view of Lemma~\ref{L:comple1}, we derive
\begin{equation*}
\omega_n < \left(1+O(\frac{1}{n})\right) \,\mathbb{E}[\omega_n]
<\left(1-\delta_k(\epsilon_n)+O(\frac{1}{n})\right)\vert
\Gamma_n\vert \quad \text{\rm a.s.,}
\end{equation*}
whence
\begin{eqnarray*}
\vert\Gamma_{n,k}\vert
\ge\left(\delta_k(\epsilon_n)-O(\frac{1}{n})\right)\vert
\Gamma_n\vert=(1-o(1))\wp(\epsilon_n)\vert\Gamma_n\vert\quad
\text{\rm a.s..}
\end{eqnarray*}
Next we prove for $n^{-\frac{1}{3}+\delta}\le \epsilon_n < 1$ and
arbitrary but fixed $k$,
\begin{equation}\label{E:will0}
\vert\Gamma_{n,k}\vert\le (1+o(1))\wp(\epsilon_n)
\, \vert\Gamma_n\vert\qquad \text{\rm  a.s.}
\end{equation}
Let $W_n=U_n(\frac{1}{2})=\{r\in \Gamma(S_n,T_n)\mid\vert C_{r}\vert
< n^{1/2}\}$,
where $C_{r}$ denotes a component containing $r$. Obviously,
$\Gamma_{n,k} \subset \Gamma_n\setminus
W_n$, whence it suffices to prove
\begin{equation}\label{E:will1}
\vert W_n \vert \ge  \left[1-(1+o(1))\wp(\epsilon_n)\right] \,
\vert\Gamma_n\vert\qquad \text{\rm  a.s.}
\end{equation}
For this purpose we follow \cite{Bollobas:91} and consider a certain
branching process in the $(n-1)$-regular rooted tree $T_{r^*}$.
Here the r.v. $\xi_r^*$ of
the rooted vertex $r^*$ is $\mbox{Bi}(n-1,\lambda_n)$ distributed while the
r.v. of any other vertex $r$ has the distribution
$\mbox{Bi}(n-2,\lambda_n)$. Let $C_{r^*}$ denote the component generated by
this branching process. The idea here is to relate $C_{r^*}$ with
its image under a covering map, i.e.~a specific $\Gamma_n$-component
containing $r$, denoted by $C_r$. \\
Using the linear ordering on $\Gamma(S_n,T_n)$, one can specify a unique
procedure on how to generate an acyclic connected $\Gamma(S_n,T_n)$-subgraph
of size $<n^{1/2}$, denoted by $H_{r}^\dagger$ \cite{Bollobas:91}.
Let $S$ be a stack. We initialize by setting $H^{\dagger}_r=\{r\}$.
Then we select the $r$-neighbors in $\Gamma(S_n,T_n)$, one by one, in
increasing order, with probability $\lambda_n$. For each selected
neighbor $r_i$, we {\sf (a)} put the corresponding edge $\{r,r_i\}$
into $S$, {\sf (b)} add $r_i$ to $H^{\dagger}_r$ and {\sf (c)} check
condition {\sf (h1)} ``$\vert H^{\dagger}_r\vert=n^{\frac{1}{2}}$''.
If {\sf (h1)} holds we stop, otherwise we proceed examining the next
$r$-neighbor. Suppose
{\sf (h1)} does not hold and all $r$-neighbors have been examined. \\
If $S$ is empty, we stop. Otherwise we proceed inductively as follows:
we remove the first element, $\{r,w\}$ from $S$ and consider the
$w$-neighbors, except $r$, one by one, in increasing order.
For each selected $w$-neighbor, $x$, we
{\sf (a)} insert the edge $\{w,x\}$ into the back of $S$
{\sf (b)} add $x$ to $H^\dagger_r$ and {\sf (c)} check condition {\sf (h1)}
``$\vert H^{\dagger}_r\vert=n^{\frac{1}{2}}$''and {\sf (h2)}
``$H^{\dagger}_r$
contains a cycle''.
In case {\sf (h1)} or {\sf (h2)} holds we stop.
Otherwise, we continue examining $w$-neighbors in increasing order
until all $w$-neighbors are considered.
If $S$ is empty we stop and otherwise we consider the next element from
$S$ and iterate the process.\\
Consequently we have by construction
\begin{equation}
\forall m\le n^{\frac{1}{2}};\quad
\mathbb{P}\left(\vert H_{r}^\dagger\vert < m \mbox{ and } H^\dagger_r
\mbox{ is a acyclic}\right) \le\mathbb{P}\left(\vert C_{r^*}\vert <
m\right),
\end{equation}
where the discrepancy between $\mathbb{P}\left(\vert
H^\dagger_{r}\vert < m \mbox{ and } H^\dagger_r \mbox{ is a
acyclic}\right)$ and $\mathbb{P}\left(\vert C_{r^*}\vert <
m\right)$ lies in those events for which a $\le$-compatible covering map from
$T_{r^*}$ into $\Gamma(S_n,T_n)$, mapping $r^*$ into $r$, produces
a cycle in $\Gamma(S_n,T_n)$. The latter is bounded from above by the
probability
$\mathbb{P}\left(H_{r}^\dagger\mbox{ contains a cycle}\right)$. Therefore,
\begin{equation}\label{E:retrees}
\forall m\le n^{\frac{1}{2}};\quad
\mathbb{P}\left(\vert H_{r}^\dagger\vert < m \mbox{ and } H^\dagger_r
\mbox{ is a acyclic}\right) \ge \mathbb{P}\left(\vert C_{r^*}\vert <
m\right)-\mathbb{P}\left(H_{r}^\dagger\mbox{ contains a cycle}\right).
\end{equation}
We proceed by computing $\mathbb{P}\left(\vert C_{r^*}\vert < m\right)$
and $\mathbb{P}(H_r^\dagger\ \text{\rm contains a cycle})$. \\
{\it Claim $1$}:\cite{Bollobas:91} there exists some $\kappa>0$ such that
\begin{eqnarray}
\mathbb{P}(\vert C_{r^*}\vert<
n^{1/2}) &\ge &
1-\pi_0(\epsilon_n)-o(e^{-\kappa\,n^{1/2}}).
\end{eqnarray}
To prove the claim we compute
\begin{eqnarray*}
\mathbb{P}(n^{1/2}\le \vert C_{r^*}\vert < \infty) &=& \sum_{i\ge
n^{1/2}}\mathbb{P}(\vert C_{r^*}\vert=i)\\
& = & \sum_{i\ge n^{1/2}} (1 + o(1))\cdot \frac{(\lambda_n\cdot
(n-2))^{i-1}}{i\sqrt{2\pi i}}
\left[\frac{(n-2)(1-\lambda_n)}{(n-3)}\right]^{ni-3i+2}\\
& \le & \sum_{i\ge n^{1/2}}\left[(1+\epsilon_n)e^{-\epsilon_n}
\right]^i\le \sum_{i\ge n^{1/2}} c(\epsilon)^i =o(e^{-\kappa n^{1/2}}),
\end{eqnarray*}
where $0<c(\epsilon)<1$ and
\begin{eqnarray}\label{E:ts}
\mathbb{P}(\vert C_{r^*}\vert=i)=(1+o(1))\cdot\frac{(\lambda_n\cdot
(n-2))^{i-1}}{i\sqrt{2\pi
i}}\left[\frac{(n-2)(1-\lambda_n)}{(n-3)}\right]^{ni-3i+2},
\end{eqnarray}
where $i=i(n)\rightarrow\infty$ as $n\rightarrow\infty$ is due to
\cite{Bollobas:91}.
We accordingly derive
\begin{eqnarray}
\mathbb{P}(\vert C_{r^*}\vert <  n^{1/2}) &= &
\mathbb{P}(\vert C_{r^*}\vert <\infty) -
\mathbb{P}(n^{1/2} \le  \vert C_{r^*}\vert < \infty)\nonumber\\
\label{E:lowertree}
&\ge&
1-\underbrace{\wp(\epsilon_n)}_{=\pi_0(\frac{1+\epsilon_n}{n-1})}-
o(e^{-\kappa n^{1/2}}),
\end{eqnarray}
where $\pi_0(\frac{1+\epsilon_n}{n-1})=\wp(\epsilon_n)=
\mathbb{P}(\vert C_{r^*}\vert =\infty)$ is the
survival probability of the branching process in $T_{r^*}$, which constructs
the component rooted in $r^*$, see Lemma~\ref{C:1}.\\
{\it Claim $2$}:
$\mathbb{P}(H_r^\dagger\ \text{\rm contains a cycle})\le
O(n^{-\frac{1}{2}})$.\\
Let $\ell$ denote the length of a cycle, $\mathcal{O}_{\ell}$, generated by
$H_r^{\dagger}$.  We first notice that $\mathcal{O}_{\ell}$ contains at
most $\lfloor\frac{\ell}{2}\rfloor$ distinct $T_n$-elements.
Otherwise $\mathcal{O}_{\ell}=(\sigma_s)_{s=1}^{\ell}$ contains
$\lfloor\frac{\ell}{2}\rfloor+1$ distinct $T_n$-transpositions and
consequently there exists at least one transposition $\sigma_{t}=(i\,j)\in
\mathcal{O}_{\ell}$ that occurs only once.
Then we conclude, using $\prod_{s=1}^{\ell}\sigma_s=1$,
\begin{equation*}
(i\, j)\in \langle T_n\setminus \{(i\, j)\}\rangle,
\end{equation*}
which is impossible since $T_n$ is a minimal generating set. Let $N$
be the number of distinct transpositions in
$\mathcal{O}_{\ell}$ and $a_s$ be
the multiplicity of $s$-th distinct transposition. We then have
$a_s\ge 2$ for $1\le s\le N$ and
$N\le \lfloor\frac{\ell}{2}\rfloor$. We notice that the number of such
cycles $\mathcal{O}_{\ell}$, that contain a fixed vertex is bounded from
above by
\begin{eqnarray*}
&&
\binom{n-1}{N}\cdot\frac{\ell!}{a_1!\cdot a_2!\cdots a_N!}\\
&\le&\binom{n-1}{N}\cdot\frac{{\ell!}}{2^N}\le
\left(\frac{n-1}{2}\right)^N\frac{{\ell!}}{N!}
\le
\left(\frac{n-1}{2}\right)^{{\lfloor\frac{\ell}{2}\rfloor}}
{\frac{\ell!}
{(\lfloor\frac{\ell}{2}\rfloor)!}}=
{O\left(\frac{\ell(n-1)}{e}\right)^{\lfloor\frac{\ell}{2}\rfloor}}
\end{eqnarray*}
We next distinguish the cases of whether or not $\mathcal{O}_{\ell}$
contains $r$. Let us first assume $r\not\in\mathcal{O}_{\ell}$. Then all
vertices except of the lastly added vertex $w$, have been examined
only once while $w$ has been examined for at most $n^{\frac{1}{2}}-1$
times. Therefore the probability of $\mathcal{O}_{\ell}$ is bounded by
\begin{eqnarray*}
\le n^{\frac{1}{2}}\cdot
\ell\cdot\left(\frac{n-1}{2}\right)^{\lfloor\frac{\ell}{2}\rfloor}
\frac{\ell!}{(\lfloor\frac{\ell}{2}\rfloor)!}\cdot
\left({\frac{2}{n-1}}\right)^{\ell-1} \frac{2}{n-1}\cdot
\left(n^{\frac{1}{2}}-1\right) =O\left(\ell
n\cdot\left(\frac{4\ell}{e(n-1)}\right)^{\lfloor\frac{\ell}{2}\rfloor}\right).
\end{eqnarray*}
Taking the sum over all possible values $4\le \ell\le n^{\frac{1}{2}}$, we
observe that the probability of the event that $H^\dagger_r$
contains such a cycle, is at most $O(n^{-1})$. \\
Suppose next $r\in \mathcal{O}_{\ell}$. Then $r$ has by construction never
been examined. The lastly added vertex (the one leading to the cycle and
therefore to the halting of the process) has been examined at most
$n^{\frac{1}{2}}-1$ times and all other vertices contained in
$\mathcal{O}_{\ell}$ have been examined only once.
Therefore the probability of $\mathcal{O}_{\ell}$ is bounded by
\begin{eqnarray*}
\le \ell\cdot\left(\frac{n-1}{2}\right)^{\lfloor\frac{\ell}{2}\rfloor}
\frac{\ell!}{(\lfloor\frac{\ell}{2}\rfloor)!}\cdot
\left({\frac{2}{n-1}}\right)^{\ell-2} \frac{2}{n-1}\cdot
\left(n^{\frac{1}{2}}-1\right) =O\left(\ell
n^{\frac{3}{2}}\cdot\left(\frac{4\ell}{e(n-1)}\right)^{
\lfloor\frac{\ell}{2}\rfloor}\right).
\end{eqnarray*}
Taking the sum over $4\le \ell\le n^{\frac{1}{2}}$, we conclude that the
probability of the event that $H^\dagger_r$ contains a cycle that contains
$r$, is at most $O(n^{-\frac{1}{2}})$ and Claim $2$ follows. \\
{\it Claim $3$:}
\begin{equation}
\mathbb{P}\left(\vert C_{{r}}\vert <n^{\frac{1}{2}} \right) \ge
1-(1+o(1)) \wp(\epsilon_n).
\end{equation}
Let $D_r$ be a tree containing $r$ of size $<n^{\frac{1}{2}}$ in
$\Gamma_n$. Since there is only one way by which the procedure
$H_r^\dagger$ can generate $D_r$ we have
\begin{equation}\label{E:smc001}
\mathbb{P}\left(C_{r}=D_r\right) \ge
                     \mathbb{P}\left(H^\dagger_{r}=D_r \right).
\end{equation}
Consequently, taking the sum over all such trees we obtain
\begin{equation}\label{E:smc00}
\mathbb{P}\left(\vert C_{r}\vert <
n^{\frac{1}{2}} \mbox{ and } \text{\rm
$C_r$ is a tree}\right)\ge \mathbb{P}\left(\vert H^\dagger_{r}\vert
< n^{\frac{1}{2}} \mbox{ and } \text{\rm
$H^\dagger_r$ is acyclic}\right) .
\end{equation}
According to eq.~(\ref{E:retrees}), Claim $1$, Claim $2$ and
$\wp(\epsilon_n)\ge n^{-1/3+\delta}$ we conclude
\begin{equation*}
\mathbb{P}\left(\vert H^\dagger_{r}\vert<
n^{\frac{1}{2}} \mbox{ and } \text{\rm
$H^\dagger_r$ is acyclic }\right) \ge 1-(1+o(1)) \wp(\epsilon_n).
\end{equation*}
Accordingly we arrive at
\begin{eqnarray*}
\mathbb{P}\left(\vert C_{r}\vert <
n^{\frac{1}{2}} \right) & \ge &
\mathbb{P}\left(\vert C_{r}\vert <
n^{\frac{1}{2}}  \mbox{ and } \text{\rm $C_r$ is a tree} \right) \\
& \ge & \mathbb{P}\left(\vert H^\dagger_{r}\vert<
n^{\frac{1}{2}} \mbox{ and }
\text{\rm $H^\dagger_r$ is acyclic}\right) \\
&\ge & 1-\wp(\epsilon_n)-o(e^{-\kappa n^{\frac{1}{2}}})-O(n^{-\frac{1}{2}})\\
&\ge & 1-(1+o(1)) \wp(\epsilon_n)
\end{eqnarray*}
and Claim $3$ is proved.
By linearity of expectation, we have $(1-(1+o(1))\wp(\epsilon_n))
\vert \Gamma_n\vert  \le \mathbb{E}[\vert W_n\vert]$ and
according to Lemma~\ref{L:comple1},
$(1-O(n^{-1}))\,\mathbb{E}[\vert W_n\vert]< \vert W_n\vert$
a.s..
In view of $n^{-1}=o(\wp(\epsilon_n))$ we have therefore proved
eq.~(\ref{E:will1})
\begin{equation*}
(1-(1+o(1))\, \wp(\epsilon_n))\,\vert \Gamma_n\vert \le \vert W_n\vert
\qquad \text{\rm a.s.}
\end{equation*}
and the proof of lemma is complete.
\end{proof}

\section{The main theorem}\label{S:5}

We show in this section that the unique giant component forms within
$\Gamma_{n,k}$ for two reasons: first, for given $\delta$, any
$\Gamma_{n,k}$-vertex is {\it a priori} contained in a subcomponent of size
$\ge M_k(n)$, see eq.~(\ref{E:polysize}), limiting the number of ways by
which $\Gamma_{n,k}$-splits
can be chosen and second there are many independent paths connecting large
$\Gamma(S_n,T_n)$-subsets. We first prove Lemma~\ref{L:dense-01} according
to which $\Gamma_{n,k}$ is ``almost'' $2$-dense in $\Gamma(S_n,T_n)$.

\begin{lemma}\label{L:dense-01}
Let $k\in \mathbb{N}$ and $\Delta_k=\left[\frac{k}{2(k+1)}\right]^{2}/{2}$, $\lambda_n=\frac{1+\epsilon_n}{n-1}$ where
$\epsilon_n\ge n^{-\frac{1}{3}+\delta}$ for some $\delta>0$ and let furthermore
$A_{\delta}=\left\{v\mid |d(v,2)\cap\Gamma_{n,k}|<\frac{1}{2}\Delta_k \cdot
n^{\delta}\right\}$.
Then $\mathbb{P}(v\in A_{\delta})\le\exp(-\frac{1}{8}\Delta_k\cdot
n^{\delta})$ and there exists some $0<\rho_k<\frac{1}{8}\Delta_k$
for arbitrary but fixed $k$, such that
\begin{equation*}
\vert A_{\delta}\vert\le n!e^{-\rho_k n^{\delta}}\quad\mbox{a.s..}
\end{equation*}
\end{lemma}
\begin{proof}
We consider now the action of the transpositions
\begin{equation*}
A_{k+1}=\left\{(v_j^{k+1}\,s_j^{k+1})\in T_n \mid 1\le
j\le\ell_n\right\}
\end{equation*}
where
$w_{j}^{(k+1)}=(v_j^{k+1}\,s_j^{k+1})=(v_{r_n-1+j+k\mu_n}\,s_{r_n-1+j+k\mu_n})$
and $\ell_n=\lfloor\frac{k}{2(k+1)}
n^{\frac{2}{3}}\rfloor$, see eq.~(\ref{E:part})
and set
\begin{equation*}
d^{(k+1)}(v,2)=\{v\cdot w_{i}^{(k+1)}\cdot w_{j}^{(k+1)}\vert 1\le
i<j\le \ell_n\}.
\end{equation*}
We proceed by
establishing a lower bound on the cardinality of $d^{(k+1)}(v,2)$.
Since $T_n$ is a minimal generating set, any sequence of distinct
$T_n$-transpositions is acyclic. Therefore
\begin{eqnarray*}
\vert d^{(k+1)}(v,2) \vert&\ge& \binom{\ell_n}{2}
=\frac{n^{\frac{4}{3}}}{2}\cdot\left[\frac{k}{2(k+1)}\right]^{2}\cdot
(1-o(1)).
\end{eqnarray*}
Let $\Delta_k=\left[\frac{k}{2(k+1)}\right]^{2}/{2}$ and $Z(v)$ be
the r.v. counting the number of vertices contained in the set
$d^{(k+1)}(v,2)\cap\Gamma_{n,k}$, whose subcomponents are
constructed in Lemma~\ref{L:expand2}. We immediately compute
\begin{equation*}
\mathbb{E}(Z(v))\ge \lambda_n\cdot \delta_k(\epsilon_n)\cdot\vert
d^{(k+1)}(v,2) \vert \sim
\Delta_k\,n^{\frac{4}{3}}\cdot\frac{1+\epsilon_n}{n-1}
\cdot\wp(\epsilon_n)(1-e^{-\beta_{k,n}\theta_{n,k}})
\ge \Delta_k\cdot n^{\delta}.
\end{equation*}
The key observation is the following: the construction of the
Lemma~\ref{L:expand2}-subcomponents did not involve any labels
$v_{r_n-1+j+k\mu_n}$, i.e.~any two such subcomponents remain
vertex-disjoint. Therefore the r.v.~$Z(v)$ is a sum of {\it
independent} indicator r.vs.~and Chernoff's large deviation
inequality, eq.~(\ref{E:lde}), \cite{Chernoff:52} implies
\begin{equation}
\mathbb{P}(v\in A_{\delta})=\mathbb{P}\left(
Z(v) <
\frac{1}{2}\,\Delta_k\cdot
n^{\delta}\right)\le \exp(-\frac{1}{8}\Delta_k\cdot n^{\delta}).
\end{equation}
Consequently, the expected number of vertices contained in $A_\delta$ is
bounded by $n!\exp(-\frac{1}{8}\Delta_k\cdot n^{\delta})$.
Now Markov's inequality \cite{Ross},
$$
\mathbb{P}(X>\,t\mathbb{E}(X))\le 1/t, \quad t>0,
$$
guarantees $\vert A_\delta\vert \le n!\cdot e^{-\rho_k n^\delta}$
a.s.~for any $0<\rho_k<\frac{1}{8}\Delta_k$ and arbitrary, fixed $k$
and the lemma follows.
\end{proof}

Next we show that there exist many vertex disjoint paths between
$\Gamma_{n,k}$-splits of sufficiently large size. The proof is
analogous to Lemma~$7$ in \cite{Reidys:08rand}. We remark that
Lemma~\ref{L:split1} does not use an isoperimetric inequality
\cite{Harper:66b}. It only employs a generic estimate of the vertex
boundary in Cayley graphs due to Aldous \cite{Aldous:87,Babai:91b}.

\begin{lemma}\label{L:split1}
Let $(S,T)$ be a vertex-split of $\Gamma_{n,k}$ with the properties
\begin{equation}\label{E:prop}
\exists\,0< \rho_0\le \rho_1<1;\quad (n-2)!\le \vert S\vert = \rho_0
\vert \Gamma_{n,k}\vert  \quad \text{\rm and}\quad (n-2)!\le
\vert T \vert = \rho_1 \vert \Gamma_{n,k}\vert .
\end{equation}
Then there exists some $c>0$ such that a.s.~$d(S)$ is connected to
$d(T)$ in $\Gamma(S_n,T_n)$ via at least
\begin{equation}
c\, (n-5)!/(n-1)^7
\end{equation}
vertex disjoint (independent) paths of length $\le 3$.
\end{lemma}
\begin{proof}
We distinguish the cases $\vert \text{\sf B}(S,2)\vert\le
\frac{2}{3} \,n!$ and $\vert\text{\sf B}(S,2)\vert>\frac{2}{3}\,
n!$. In the former case, we employ the generic estimate of vertex
boundaries in Cayley graphs \cite{Aldous:87}
\begin{equation}\label{E:gen}
\vert {\sf d}(S) \vert
\ge \frac{1}{{\rm diam}(\Gamma(S_n,T_n))}\cdot\vert S \vert
\left(1-\frac{\vert S\vert}{n!}\right).
\end{equation}
In view of eq.~(\ref{E:prop}) and Lemma~\ref{L:gut}, eq.~(\ref{E:gen}) implies
\begin{equation}
\exists \,d_1>0;\quad \vert {\sf d}({\sf B}(S,2))\vert
\ge\frac{d_1}{n^2}\cdot\vert{\sf B}(S,2)\vert \ge d_1 \cdot (n-4)!  .
\end{equation}
According to Lemma~\ref{L:dense-01}, a.s.~all but $\le n!\,
e^{-\rho_k n^\delta}$ permutations are within
distance $2$ to some $\Gamma_{n,k}$-vertex, whence
\begin{equation}
\vert \text{\sf d}(\text{\sf B}(S,2))\cap \text{\sf B}(T,2)\vert\ge
d_2\cdot (n-4)! \quad \text{\rm a.s..}
\end{equation}
Let $\beta_2\in \text{\sf d}(\text{\sf B}(S,2))\cap \text{\sf B}(T,2)$. Then
there exists a path $(\alpha_1,\alpha_2,\beta_2)$ such that $\alpha_1\in {\sf
d}(S)$, $\alpha_2\in{\sf d}(B(S,1))$. We distinguish the cases
\begin{equation}
\vert {\sf d}(\text{\sf B}(S,2))\cap {\sf d}(\text{\sf B}(T,1))
\vert\ge d_{2,1} \,(n-4)! \quad \text{\rm and}\quad \vert {\sf
d}(\text{\sf B}(S,2))\cap \text{\sf B}(T,1) \vert\ge d_{2,2}
\,(n-4)!.
\end{equation}
For $\vert {\sf d}(\text{\sf B}(S,2))\cap {\sf d}(\text{\sf B}(T,1))
\vert\ge d_{2,1} \,(n-4)!$, we consider the set
\begin{equation*}
T^*=\{\beta_1 \in \text{\sf d}(T)\mid  d(\beta_1,\beta_2)=1,\text{\rm
for some $\beta_2\in {\sf d}(\text{\sf B}(T,1))$}\}.
\end{equation*}
Evidently, at most $n-1$ elements in $\text{\sf d}(T)$ can be connected to a
fixed $\beta_2$, whence
$$
\vert T^*\vert\ge \frac{1}{2}d_{2,1} \, (n-5)! .
$$
Let $T_1\subset T^*$ be some maximal set such that any pair of $T_1$-vertices
$(\beta_1,\beta_1')$ has at least distance $d(\beta_1,\beta_1')> 6$.
Then $\vert T_1\vert > \vert T^*\vert/(n-1)^7$ since
$\vert \text{\sf B}(v,6)\vert<\sum_{i=1}^6(n-1)^i<(n-1)^7$.
Any two of the paths from $\text{\sf d}(S)$
to $T_1\subset \text{\sf d}(T)$ are of the form
$(\alpha_1,\alpha_2,\beta_2,\beta_1)$ and vertex disjoint since each
of them is contained in $\text{\sf T}(\beta_1,3)$.
Accordingly there are a.s.~at least
\begin{equation}
\frac{1}{2}d_{2,1} \,(n-5)!  /(n-1)^7
\end{equation}
vertex disjoint paths connecting $\text{\sf d}(S)$ and $\text{\sf
d}(T)$. In case of $\vert {\sf d}(\text{\sf B}(S,2))\cap \text{\sf B}(T,1)
\vert\ge \, d_{2,2} \, (n-3)!$ we analogously conclude, that
there exist a.s.~at least
\begin{equation}
d_{2,2} \,(n-4)! /(n-1)^5
\end{equation}
vertex disjoint paths of the form $(\alpha_1,\alpha_2,\beta_2)$
connecting $\text{\sf d}(S)$ and $\text{\sf d}(T)$.\\
It remains to consider the case $\vert \text{\sf B}(S,2)\vert >
\frac{2}{3}\cdot n!$. By construction both: $S$ and $T$ satisfy
eq.~(\ref{E:prop}), whence we can, without loss of
generality assume that also $\vert \text{\sf B}(S,2)\vert >
\frac{2}{3}\cdot n!$ holds. But then
$$
\vert \text{\sf B}(S,2)\cap \text{\sf B}(T,2)\vert > \frac{1}{3}\,n!
$$
and for each $\alpha_2\in\text{\sf B}(S,2)\cap \text{\sf B}(T,2)$ we
select $\alpha_1\in \text{\sf d}(S)$ and $\beta_1\in \text{\sf
d}(T)$.
We derive in analogy to the previous arguments that there
exist a.s.~at least
\begin{equation}
d_2 \, (n-2)! /(n-1)^5
\end{equation}
pairwise vertex disjoint paths of the form
$(\alpha_1,\alpha_2,\beta_1)$ and the proof of the lemma is
complete.
\end{proof}


{\bf Proof of Theorem~\ref{T:main}.}
To prove the theorem we employ an argument due to Ajtai {\it et al.}
\cite{Ajtai:82} originally used for $n$-cubes and independent edge-selection.
We proceed along the lines of \cite{Reidys:08rand} and select the
$\Gamma(S_n,T_n)$-vertices in two distinct randomizations.\\
Let $x_1,x_2>1$
such that $\frac{1}{x_1}+\frac{1}{x_2}=1$. First we select with
probability $\frac{1+\epsilon_n/x_1}{n}$ and second with probability
$\frac{\epsilon_n}{x_2 \cdot n}$. The probability of not being
chosen in both rounds is given by
\begin{equation*}
\left(1-\frac{1+\epsilon_n/x_1}{n}\right)
\left(1-\frac{\epsilon_n}{x_2\cdot n}\right)\ge
1-\frac{1+\epsilon_n}{n},
\end{equation*}
whence it suffices to prove that after the second randomization there
exists a giant component with the property $\vert C_{n}^{(1)}\vert\sim
\vert\Gamma_{n,k}\vert$.\\
After the first randomization each $\Gamma(S_n,T_n)$-vertex has been
selected with probability $\frac{1+\epsilon_n/x_1}{n}$ and according
to Lemma~\ref{L:size1}, we have
\begin{equation}
\vert \Gamma_{n,k}(x_1)\vert\sim \wp(\epsilon_n/x_1)\,
\vert\Gamma_n(x_1)\vert \quad \text{\rm a.s.,}
\end{equation}
where $\Gamma_n(x_1)\subset \Gamma_n$. Suppose
$\Gamma_{n,k}(x_1)$ contains a ``large'' component, $S$. To be precise
a component $S$ of size
$$
(n-2)! \le \vert S\vert \le (1-b)\, \vert \Gamma_{n,k}(x_1)\vert,
\quad \text{\rm where } b>0.
$$
Then there exists a split of $\Gamma_{n,k}(x_1)$, $(S,T)$,
satisfying the assumptions of Lemma~\ref{L:split1}.
We observe that Lemma~\ref{L:expand2} limits the
number of ways these splits can be constructed. Recall
(eq.~(\ref{E:polysize}))
$$
M_k(n) = \frac{1}{2^{k+2}}\cdot\left[\frac{1}{4k(k+1)}\right]^k
        \cdot n^{\frac{2}{3}+k\delta}.
$$
Obviously, there are at most $2^{n!/M_k(n)}$
ways to select $S$ of such a split.
Now we employ Lemma~\ref{L:split1}. In view of $(n-2)!\le \vert S\vert$,
Lemma~\ref{L:split1} implies that there exists some
$c>0$ such that a.s.~$d(S)$ is connected to $d(T)$ in $\Gamma(S_n,T_n)$
via at least $c\cdot n!/n^{12}\le c\cdot\vert S \vert/n^{10}$ vertex disjoint
paths of length $\le 3$.\\
We next perform the second randomization and select
$\Gamma(S_n,T_n)$-vertices with probability
$\frac{\epsilon_n/x_2}{n}$. None of the above $c\cdot\vert S
\vert/n^{10}$ paths can be selected during this process. Since any two
paths are vertex disjoint the expected number of such splits is, by
linearity of expectation, less than
\begin{equation}
2^{n!/M_k(n)} (1-(\epsilon_n/x_2n)^4)^{\frac{c\cdot n!}{n^{12}}}
\le 2^{n!/M_k(n)} e^{-c' n!/n^{16}}
\quad\text{for some $c,c'>0$}.
\end{equation}
Accordingly, choosing $k$ sufficiently large the expected number of these
$\Gamma_{n,k}(x_1)$-splits tends to zero, i.e.~for any $k\ge k_0\in
\mathbb{N}$ there exists a.s.~no two component split $(S,T)$ of
$\Gamma_{n,k}(x_1)$ with the property
$\rho_0\vert\Gamma_{n,k}(x_1)\vert=\vert S\vert\le
\vert T\vert$ .
Consequently, there exists some subcomponent $C_n(x_1)$
with the property
\begin{equation*}
\vert C_n(x_1) \vert=\vert \Gamma_{n,k}(x_1)\vert \sim
\wp(\epsilon_n/x_1)\, \vert \Gamma(x_1)\vert \quad
\text{\rm a.s.,}
\end{equation*}
obtained by the merging of the subcomponents of size
$\ge M_k(n)$ generated during the first randomization
via the paths selected during the second. Since $\wp(\epsilon_n/x_1)$ is
continuous in the parameter $\epsilon_n/x_1$, see eq.~(\ref{E:it}),
we derive, for $x_1$ tending to $1$
\begin{equation}
\vert C_n^{(1)} \vert=\lim_{x_1\to 1}\vert C_n(x_1) \vert \sim \wp(\epsilon_n)
\vert \Gamma_n\vert \quad \text{\rm
a.s.}
\end{equation}
It remains to prove uniqueness.
Any other largest component, $\tilde{C}_n$, is necessarily contained in
$\Gamma_{n,k}$. However, we have just proved $\vert C_n^{(1)} \vert
\sim \wp(\epsilon_n)\vert \Gamma_n\vert$ and according to
Lemma~\ref{L:size1},
$\wp(\epsilon_n)\vert \Gamma_n\vert\sim \vert \Gamma_{n,k}\vert$.
Therefore
$\vert\tilde{C}_n\vert=o(\vert C_n^{(1)}\vert)$, whence $C_n^{(1)}$
is unique.\hfil  $\square$



{\bf Acknowledgments.}
Emma Y.~Jin would like to thank the Alexander von Humboldt Foundation for
their support. We are grateful to Fenix W.D.~Huang and Rita R.~Wang for
their help.

\end{document}